\documentclass[12pt,a4paper]{article}

\usepackage{amssymb}
\usepackage{amsthm}
\usepackage{amsmath}
\usepackage[colorlinks=true,allcolors=black]{hyperref}
\usepackage{enumerate}
\usepackage{tikz}
\usetikzlibrary{positioning}
\usepackage{tkz-tab}
\usepackage{caption}
\usepackage{subcaption}

\newtheorem{theorem}{Theorem}[section]

\newtheorem{lemma}[theorem]{Lemma}
\newtheorem{proposition}[theorem]{Proposition}

\begin{document}

\title{On finite groups whose power graphs satisfy certain connectivity conditions}
\author{Ramesh Prasad Panda\thanks{Department of Mathematics, School of Advanced Sciences, VIT-AP University, Amaravati, PIN-522237, Andhra Pradesh, India.}}

\date{}

	\maketitle
	
		\makeatletter{\renewcommand*{\@makefnmark}{}
		\footnotetext{Email address: {\tt rameshprpanda@gmail.com}}
		\footnotetext{ORCID ID: \href{https://orcid.org/0000-0001-7901-9828}{0000-0001-7901-9828}}

\begin{abstract}
Consider a graph $\Gamma$. A set $ S $ of vertices in $\Gamma$ is called a {cyclic vertex cutset} of $\Gamma$ if $\Gamma - S$ is disconnected and has at least two components containing cycles. If $\Gamma$ has a cyclic vertex cutset, then it is said to be {cyclically separable}. The {cyclic vertex connectivity} is the minimum cardinality of a cyclic vertex cutset of $\Gamma$.
The power graph $\mathcal{P}(G)$ of a group $G$ is the undirected simple graph with vertex set $G$ and two distinct vertices are adjacent if one of them is a positive power of the other. If $G$ is a cyclic, dihedral, or dicyclic group, we determine the order of $G$ such that $\mathcal{P}(G)$ is cyclically separable. Then we characterize the equality of vertex connectivity and cyclic vertex connectivity of $\mathcal{P}(G)$ in terms of the order of $G$.

\medskip



\end{abstract}

\section{Introduction}

The notion of graphs defined on groups has been in existence since the work of Cayley \cite{cayley1878desiderata} in 1878. These graphs came to known as Cayley graphs, and have been widely studied in literature, see \cite{li2002,Witte}. In the last few decades, several other graphs associated with groups, such as commuting graphs \cite{Brauer}, prime graphs \cite{Williams}, and conjugacy class graphs \cite{Bertram2}, were introduced. These graphs have been studied extensively by researchers, and have various applications \cite{Bianchi,cooperman,Hayat,kelarev_data}. Kelarev and Quinn \cite{kelarevquinn} introduced power graphs with similar interest. The \emph{power graph} of a group $ G $,
denoted by $\mathcal{P}(G)$, is the simple undirected graph with vertex set  $ G $, and two vertices  are adjacent in the graph if one of them is a positive power of the other in $G$. In recent years, various aspects of power graphs have been investigated; see \cite{Abawajy,Kumar} and the references therein.

Let $\Gamma$ be an undirected and simple graph. The \emph{vertex connectivity} $\kappa(\Gamma)$ of $\Gamma$ is the minimum number of vertices whose deletion either disconnects the graph or reduces it to a trivial graph. A set $ S $ of vertices is called a \emph{vertex cutset} of $\Gamma$ if $ \Gamma - S $ is disconnected. A vertex cutset is \emph{minimal} if none of its proper subsets disconnects $\Gamma$. A vertex cutset of $\Gamma$ with minimum number of elements is called a \emph{minimum vertex cutset} of $\Gamma$. For non-complete graphs, the vertex connectivity of $ \Gamma $ is the number of elements of a minimum vertex cutset of $ \Gamma $. A \emph{cyclic vertex cutset} is a vertex cutset $ S $ such that $ \Gamma - S $ contains at least two components, each of which includes a cycle. A graph possessing a cyclic vertex cutset is said to be \emph{cyclically separable}. The \emph{cyclic vertex connectivity}, denoted by $ c\kappa(\Gamma) $, is the minimum cardinality of all cyclic vertex cutsets in $ \Gamma $. If no such cutset exists, $ c\kappa(\Gamma) $ is considered to be infinite. Similarly, the concept of edge connectivity and cyclic edge connectivity are defined by considering the  edge deletion instead of vertices. 

The concept of cyclic connectivity was first introduced in the proof of Tait's well-known but incorrect conjecture in 1880 \cite{tait1880}. With this, Tait aimed to prove the four-color theorem. Later, Birkhoff \cite{birkhoff1913} utilized the concept of cyclic connectivity to reduce the four-color problem to a particular class of planar cubic graphs. Cyclic connectivity is known to have applications in areas such as integer flow conjectures \cite{zhang1997} and network reliability analysis \cite{latifi1994}. For further literature on cyclic connectivity, see \cite{liu2011,liu2022,nedela2022,robertson1984} and references therein.

Chattopadhyay and Panigrahi \cite{chattopadhyay2014} studied the vertex connectivity of the power graphs of finite cyclic, dihedral and dicyclic groups. Panda and Krishna \cite{panda2018a} and Chattopadhyay et al. \cite{chattopadhyay2019} computed the vertex connectivity $\kappa(\mathcal{P}(C_n))$ for various orders $n$ of the cyclic group $C_n$. The computation of $\kappa(\mathcal{P}(C_n))$ was further extended in \cite{chattopadhyay2020, Mukherjee2024}. In \cite{chattopadhyay-noncyclic}, Chattopadhyay et al. considered the power graph $\mathcal{P}(G)$ of a non-cyclic nilpotent group $G$ and obtained the vertex connectivity for all $G$ satisfying some conditions. Panda and Krishna \cite{panda2018b} showed that the edge connectivity and minimum degree coincide for power graphs of finite groups. Then they computed the minimum degree of power graphs of some finite groups. Panda et al. \cite{panda2021, panda2023} obtained the minimum degree of power graphs of finite cyclic and non-cyclic nilpotent groups. Furthermore, in \cite{panda2024equality}, the authors characterized the equality of vertex connectivity and minimum degree of power graphs of finite nilpotent groups. 

In \cite{panda2024}, the present author obtained the finite $p$-groups whose power graphs are cyclically separable. Then the author characterized the finite $p$-groups whose power graphs have equal vertex connectivity and cyclic vertex connectivity. In this paper, we first determine the finite cyclic, dihedral, and dicyclic groups whose power graphs are cyclically separable. Then we characterize the equality of vertex connectivity and cyclic vertex connectivity of power graphs of these groups.

Throughout, for a positive integer $n$,  we fix the prime factorization  $n = p_1^{\alpha_1}p_2^{\alpha_2} \cdots p_r^{\alpha_r}$, where $r \geq 2$, $\alpha_1,\alpha_2,\ldots, \alpha_r$ are positive integers, and $p_1<p_2<\cdots <p_r$ are primes. For any set $A$, we write $|A|$ for the number of elements in $A$.

All graphs considered in this paper are undirected and simple. Consider a graph $\Gamma$. The {\em neighbourhood} of a vertex $v$ in $\Gamma$, denoted by $N(v)$, is the set of vertices adjacent to $v$. Whereas, the \emph{degree} of $v$ in $\Gamma$, denoted by $\deg(v)$, is the number of vertices adjacent to $v$. So, $\deg(v) = |N(v)|$. In fact, for any vertex set $A$ of $\Gamma$, we can define its {\em neighbourhood} $N(A)$ as the set of all vertices which does not belong to $A$ and are adjacent to some vertex in $A$. We observe that $N(A) = \cap_{v \in A} N(v)$.

For any group $G$ and for $x \in G$, $\langle x \rangle$ is the cyclic subgroup of $G$ generated by $x$. Whereas, $[x]$ is the set of generators of $\langle x \rangle$. We denote by $C_n$ the cyclic group of order $n$. Let $x \in C_n$ with order $\text{o}(x)$.  By \cite[Lemma 2.7]{cur-2014} (also, see\cite[Lemma 3.4]{MRS-JAA}), we have the following formula for $\deg(x)$:
\begin{equation}\label{eqn-1}
	\deg(x) = \text{o}(x) + \underset{d\, | \, \frac{n}{\text{o}(x)}} \sum \phi \left( \frac{n}{d} \right) - \phi \left(\text{o}(x) \right) -1,
\end{equation}
where $\phi$ is the Euler's totient function.

 For any positive divisor $d$ of $n$, we define the following subsets and subgroups of $C_n$:

\medskip
$O_d$ = the set of elements  in $C_n$ of order $d$,

$H_d$ = the subgroup of $C_n$ of order $d$.
\medskip \\
If $x \in C_n$ is of order $d$, then $H_d = \langle x \rangle$ and $O_d = [x]$. Thus $|O_d| = \phi(d)$ and $|H_d| = d$. Note that $O_d$ is precisely the set of elements of order $d$ in $C_n$. From the definition of power graph, $O_d$ induces a clique of size $\phi(d)$ in $\mathcal{P}(C_n)$. Also, for  positive  divisors $c$ and $d$ of $n$, each element in $O_c$ is adjacent to every element in $O_d$ if and only if  $c \mid d$ or  $d \mid c$.

Note that $O_n$ is the set of all generators, and $O_1$ is the set of identity element of $C_n$. Thus each vertex in $O_n \cup O_1$ is adjacent to every other veretx in $ \mathcal{P}(C_n) $. A trivial consequence of this is the following lemma.

\begin{lemma}
	For any poitive integer $n$ and any cutset  $X$ of $\mathcal{P}(C_n)$, $O_n \cup O_1 \subseteq X$. 
\end{lemma}

Let $\Gamma$ be a graph with the vertex set $V(\Gamma)$. Then for any $S \subseteq V(\Gamma)$, we denote $\overline{S} = V(\Gamma) \setminus S$. Also, we denote $C'_n = C_n - (O_n \cup O_1)$ and $\mathcal{P}'(C_n) = \mathcal{P}(C_n) - (O_n \cup O_1)$. Then in view of the above lemma, we state the following. 

For any poitive integer $n$, $X$ is a cutset $\mathcal{P}(C_n)$ if and only if  $X \setminus (O_n \cup O_1)$ is a cutset of $\mathcal{P}'(C_n)$. Moreover, $X$ is a cyclic cutset $\mathcal{P}(C_n)$ if and only if  $X \setminus (O_n \cup O_1)$ is a cyclic cutset of $\mathcal{P}'(C_n)$.

\begin{lemma}[{\cite{panda2018a}}]
	\label{lem_minimal_cutset}
For any $1 \leq k \leq r$, $\displaystyle Y_k := O_n \cup  \bigcup_{\substack{i=1,\, i \neq k}}^{r} H_{\frac{n}{p_ip_k}}$ is a minimal cutset of $\mathcal{P}(C_n)$. In fact,
$$|Y_k| = \phi(n) + \frac{n}{p_k} - p_k^{\alpha_k-1} \phi \left(\dfrac{n}{p_k^{\alpha_k}} \right).$$
\end{lemma}

In particular, we have the following lemma.

\begin{lemma}[{\cite{panda2018a}}]\label{lem_alpha}
	For any positive integer $n$, if $r \geq 2$, then $$\kappa(\mathcal{P}(C_n)) \leq \alpha(n): = \phi(n) + \dfrac{n}{p_r} - p_r^{\alpha_{r}-1}\phi\left( \dfrac{n}{p_r^{\alpha_{r}}} \right) .$$
\end{lemma}

It was further shown in \cite{panda2018a} (also, see \cite{chattopadhyay2019}) that if $n$ has exactly two prime factors and $k = r$, then $Y_k$ is a minimum cutset of $\mathcal{P}(C_n)$. So its cardinality is the vertex connectivity of $\mathcal{P}(C_n)$, as stated below.

\begin{lemma}[{\cite{panda2018a}}]
	\label{lem2}
	If $ r=2 $, then $Y_1 = Y_2$ and it is a minimum cutset of $\mathcal{P}(C_n)$. In particular, 
	$\kappa(\mathcal{P}(C_n)) = |Y_1| = \phi(n) + p^{\alpha_1-1}q^{\alpha_2-1}$.
\end{lemma}

\section{Connectivities}

We begin with the following observation about isomorphic graphs.

\begin{lemma}
If the graphs $\Gamma_1$ and $\Gamma_2$ are isomorphic, then $\Gamma_1$ is cyclically separable if and only if $\Gamma_2$ is cyclically separable.
\end{lemma}

We first recall the cyclic separability and the equality of vertex connectivity and cyclic vertex connectivity of power graphs of finite $p$-groups. Hereafter, we refer to cyclic vertex cutsets as simply cyclic cutsets.

\begin{theorem}[{\cite{panda2024}}]\label{thm_p_group}
	For any finite $p$-group $G$, $\mathcal{P}(G)$ is cyclically separable if and only if $G$ satisfies one of the following conditions:
		\begin{enumerate}[\rm(i)]
			\item $p>3$ and $G$ is non-cyclic,
			\item $p=3$ and $G$ has at least two maximal cyclic subgroups of order greater than $3$,
			\item $p=2$ and $G$ has at least two maximal cyclic subgroups of order greater than $4$, or that $G$ has at least two maximal cyclic subgroups of order greater than $2$ with trivial intersection.
		\end{enumerate}
\end{theorem}

\begin{theorem}[{\cite{panda2024}}]\label{mainthm2}
		Let $G$ be a finite $p$-group. Then $ \kappa(\mathcal{P}(G)) = c\kappa(\mathcal{P}(G)) $ if and only if $G$ satisfies one of the following conditions:
		\begin{enumerate}[\rm(i)]
			\item $p>3$ and $G$ is non-cyclic,
			\item $p \in \{2,3\}$ and $G$ has at least two maximal cyclic subgroups of order greater than $p$ with trivial intersection.
		\end{enumerate}
	\end{theorem}

\subsection{Cyclic groups}

In this subsection, we consider first the cyclic separability and then the equality of vertex connectivity and cyclic vertex connectivity of power graphs of finite cyclic groups.

\begin{theorem}
	\label{thm1} 
For any positive integer $n$, $\mathcal{P}(C_n)$ is cyclically separable if and only if the following conditions hold:
\begin{enumerate}[\rm(i)]
\item $n$ has at least two prime factors,
\item $n \neq p_1p_2$ for any primes $p_1 < p_2$ such that $p_1 \in \{2,3\}$,
\item $n \neq 12$.
\end{enumerate} 
\end{theorem}

\begin{proof}
Let $n$ have at least three prime factors, namely $p_1 < p_2 < p_3$. Then $p_1 \geq 2 $, $ p_2 \geq 3 $, and $p_3 \geq 5$. 
Note that no element of $H_{p_1p_2}$ is adjacent to any element of $O_{p_3}$ in $\mathcal{P}(C_n)$.
 So for $S = H_{p_1p_2}^* \cup O_{p_3}$, $\overline{S}$ is cut-set of $\mathcal{P}(C_n)$. Moreover, $O_{p_3}$ induces a clique in $\mathcal{P}(C_n)$ and $|O_{p_3}| = \phi( p_3 ) \geq 4$. Hence, the component of $\mathcal{P}(C_n) - \overline{S}$ induced by $O_{p_3}$ contains a cycle. Next, $O_{p_1p_2}$ is a clique and $|O_{p_1p_2}| = \phi(p_1p_2) = \phi(p_1)\phi(p_2) \geq 2$. As $O_{p_1p_2}$ is the set of generators of $ H_{p_1p_2}$, $a \sim b$ for all $a \in O_{p_1p_2}$  and $b \in O_{p_1}$. So $O_{p_1p_2} \cup O_{p_1}$ is a clique in $\mathcal{P}(C_n)$ and $|O_{p_1p_2} \cup O_{p_1}| = \phi(p_1p_2)+\phi(p_2) \geq 4$. Hence,  the component of $\mathcal{P}(C_n) - \overline{S}$ induced by $ H_{p_1p_2}^* $ contains a cycle. Thus, $\mathcal{P}(C_n)$ is cyclically separable.
 
 Now, let $n$ have exactly two prime factors; that is $n=p_1^{\alpha_1}p_2^{\alpha_2}$ for primes $p_1 < p_2$ and positive integers $\alpha_1$ and $\alpha_2$. Note that $p_1 \geq 2$ and $ p_2 \geq 3$. If $p_1 \geq 5$, then both $O_{p_1}$ and $O_{p_2}$ are cliques of size at least $4$ in $\mathcal{P}'(C_n)$. Additionally, $a \nsim b$ for all $a \in O_{p_1}$  and $b \in O_{p_2}$. Hence, taking $S = O_{p_1} \cup O_{p_2}$, $\overline{S}$ becomes a cyclic cutset of $\mathcal{P}'(C_n)$. Next we assume that $p_1 \in \{2,3\}$. Then we have the following cases.
 
 \medskip
 
 \noindent
 {\bf Case  1.}  $ \alpha_1 = 1,$ $\alpha_2 =1$. Then $C'_n = O_{p_1} \cup O_{p_2}$   and hence $\mathcal{P}'(C_n)$ is disconnected. However, $|O_{p_1}| = p_1 -1 \leq 2$. Hence, $\mathcal{P}(C_n)$ has no cyclic cutset.
 
    \medskip

    \noindent
    {\bf Case  2.} $ \alpha_1 = 2,$ $\alpha_2 =1$. Then $n = p_1^2 \cdot p_2$ and $C'_n = O_{p^2_1} \cup O_{p_1} \cup O_{p_1 \cdot p_2} \cup O_{p_2}$. Note that $O_{{p^2_1}}$, $O_{p_1} $, $ O_{p_1 \cdot p_2} $, and  $ O_{p_2}$ are cliques of size $p_1(p_1-1)$, $p_1-1$, $(p_1-1)(p_2-1)$, and $p_2-1$, respectively. The adjacency relation between the elements of these four sets in $\mathcal{P}'(C_n)$ can be visualized in the following figure.
    
  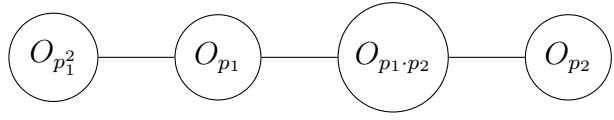
\begin{figure}[h]
  	
  	\begin{center}
 \begin{tikzpicture}[
 	roundnode/.style={circle, draw=black},
 	squarednode/.style={rectangle, draw=red!60, fill=red!5, very thick, minimum size=5mm},
 	]
 	\node[roundnode]      (maintopic)                              {$O_{p_1 \cdot p_2}$};
 	\node[roundnode]        (uppercircle)       [left=of maintopic] {$O_{p_1}$};
 	\node[roundnode]      (rightsquare)       [right=of maintopic] {$O_{p_2}$};
 	\node[roundnode]        (lowercircle)       [left=of uppercircle] {$O_{p^2_1}$};
 	
 	\draw (uppercircle.east) -- (maintopic.west);
 	\draw (maintopic.east) -- (rightsquare.west);
 	\draw (lowercircle.east) -- (uppercircle.west);
 \end{tikzpicture}
  		\caption{$\mathcal{P}'(C_n)$} \label{fig:M1}
  	\end{center}
  	
  \end{figure} 

We conclude from the figure that to disconnect $\mathcal{P}'(C_n)$, we must delete at least $O_{p_1} $ or $ O_{p_1 \cdot p_2} $. However, $\mathcal{P}'(C_n) - O_2$ is disconnected with two component induced by $O_{p^2_1}$ and $ O_{p_1 \cdot p_2} \cup O_{p_2}$. Whereas, $\mathcal{P}'(C_n) - O_{p_1 \cdot p_2}$ is disconnected with two component induced by $ O_{p_1} \cup O_{p^2_1}$ and $ O_{p_2}$. 

If $p_2  = 3$, then $p_1  = 2$. So, $|O_{p^2_1}| =  |O_{p_2}|  =2$. Hence $\mathcal{P}(C_n)$ has no cyclic cutset. Whereas, if $p_2 \geq 5$, then $ O_{p_1} \cup O_{p^2_1}$ and $ O_{p_2}$ are cliques of sizes $p^2_1 -1 \geq 3$ and $p_2 -1 > 3$, respectively. Hence  $ O_{p_1 \cdot p_2} $  is a cyclic cutset of $\mathcal{P}'(C_n)$.

  \medskip
 
 \noindent
 {\bf Case  4.} $ \alpha_1 \geq 3,$ $\alpha_2 =1$. Then for $S = O_{p_1^{\alpha_1}} \cup O_{p_1p_2} \cup O_{p_2}$,  $\overline{S}$  is a  cutset of $\mathcal{P}'(C_n)$. Further, $O_{p_1^{\alpha_1}} $ and $ O_{p_1p_2} \cup O_{p_2}$ are cliques such that $|O_{p_1^{\alpha_1}} | \geq \phi(p_1^3) \geq \phi(2^3) > 3 $ and $|O_{p_1p_2} \cup O_{p_2}| = \phi(p_1p_2) + \phi(p_2) \geq \phi(2 \cdot 3) + \phi(3)  > 3$. Hence $\overline{S}$  is a cyclic cutset of $\mathcal{P}'(C_n)$.
 
   \medskip
 
 \noindent
 {\bf Case  5.} $ \alpha_2 \geq 2$. Then for $S = O_{p_1} \cup O_{p_1p_2} \cup O_{p_2^{\alpha_2}}$,  $\overline{S}$  is a cutset of $\mathcal{P}'(C_n)$. Moreover, $O_{p_1} \cup O_{p_1p_2}$ and $O_{p_2^{\alpha_2}}$ cliques such that $|O_{p_1} \cup O_{p_1p_2}| = \phi(p_1) + \phi(p_1p_2) \geq 1+2  = 3$ and $|O_{p_2^{\alpha_2}}| \geq \phi(p_2^2) \geq \phi(3^2) > 3$. Hence $\overline{S}$  is a cyclic cutset of $\mathcal{P}'(C_n)$.
\end{proof}

We now study neighbourhoods of $[x]$ for any $x \in C_n$ in $\mathcal{P}(C_n)$.

\begin{lemma}\label{lem0}
	For any positive integer $n$ and $x \in C_n$, we have
	\begin{equation*}\label{eqn-2}
		|N([x])| = \text{o}(x) - 2 \cdot \phi \left(\text{o}(x) \right) + \underset{d\, | \, \frac{n}{\text{o}(x)}} \sum \phi \left( \frac{n}{d} \right).
	\end{equation*}
\end{lemma}

\begin{proof}
	For any $x \in C_n$, we have $N(x)  = N([x]) \cup ([x]\setminus\{x\}) $. As $N([x]) \cap ([x]\setminus\{x\}) = \emptyset$, we get 
	\begin{align*}
		|N(x)| & = |N([x])| + |[x]\setminus\{x\}|\\
		& = |N([x])| + \phi(\text{o}(x))-1.
	\end{align*}
	Hence, 
	\begin{align*}
		|N([x])| & =  |N(x)| - \phi(\text{o}(x)) + 1\\
		& = \text{o}(x) - 2 \cdot \phi \left(\text{o}(x) \right) + \underset{d\, | \, \frac{n}{\text{o}(x)}} \sum \phi \left(\frac{n}{d}\right), \, \text{ by (\ref{eqn-1}).}
	\end{align*}
\end{proof}

	\begin{lemma}\label{lem_deg_prime_power}
		For any positive integer $n$ and $x \in C_n$ of  order $p_k^{\beta}$ for some positive integer $\beta$ and $1 \leq k \leq r$, we have
		\begin{equation*}
			\displaystyle |N([x])| = p_k^{\beta} - 2 \cdot \phi \left(p_k^{\beta} \right) + (p_k^{\alpha_k} - p_k^{\beta-1}) \prod_{j=1, j \neq k}^{r} p_j^{\alpha_j}.
		\end{equation*}
	\end{lemma}

	\begin{proof}
In view of Lemma \ref{lem0}, it is enough to prove that 		\begin{align*}	
	\underset{d\, | \, { \frac{n}{p_k^{\beta}}}} \sum \phi \left( \frac{n}{d} \right) = (p_k^{\alpha_k} - p_k^{\beta-1}) \prod_{j=1, j \neq k}^{r} p_j^{\alpha_j}.
\end{align*}	
%
		We have	
		\begin{align*}	
			\underset{d\, | \, { \frac{n}{p_k^{\beta}}}} \sum \phi \left( \frac{n}{d} \right) &  =  \underset{d\, | \, {p_k^{\alpha_k - \beta} \prod_{j=1, j \neq k}^{r} p_j^{\alpha_j} }} \sum \phi \left( \frac{p_1^{\alpha_1}p_2^{\alpha_2} \cdots p_r^{\alpha_r}}{d} \right)\\
			& = \left\{ \sum_{i=0}^{\alpha_k - \beta} \phi \left(p_k^{\alpha_k-i}\right) \right\} \prod_{j=1, j \neq k}^{r} \left\{ \sum_{i=0}^{\alpha_j} \phi \left({p_j^{\alpha_j-i}}\right) \right\} \\
			& = (p_k^{\alpha_k} - p_k^{\beta-1}) \prod_{j=1, j \neq k}^{r} p_j^{\alpha_j}.
		\end{align*}
		
	\end{proof}

\begin{lemma}\label{compare_nbd}
	For positive integers $1 \leq k \leq r$, $\beta < \gamma$, if $x,y$ are elements of orders $p_k^{\beta}$ and $p_k^{\gamma}$ in $C_n$, respectively, then $|N([x])| > |N([y])|$.
	\begin{align*}
	 &|N([x])| - |N([y])|\\
	 & = p_k^{\beta} - 2 \cdot \phi \left(p_k^{\beta} \right) + (p_k^{\alpha_k} - p_k^{\beta-1}) \prod_{j=1, j \neq k}^{r} p_j^{\alpha_j} - \left[ p_k^{\gamma} - 2 \cdot \phi \left(p_k^{\gamma} \right) + (p_k^{\alpha_k} - p_k^{\gamma-1}) \prod_{j=1, j \neq k}^{r} p_j^{\alpha_j} \right]\\
	 & = - p_k^{\beta} + 2 p_k^{\beta-1} - (- p_k^{\gamma} + 2 p_k^{\gamma-1}) + (p_k^{\gamma-1} - p_k^{\beta-1}) \prod_{j=1, j \neq k}^{r} p_j^{\alpha_j}  \\
	& = p_k^{\gamma} - p_k^{\beta} - 2( p_k^{\gamma-1} - p_k^{\beta-1}) + (p_k^{\gamma-1} - p_k^{\beta-1}) \prod_{j=1, j \neq k}^{r} p_j^{\alpha_j}  \\
   & = (p_k^{\gamma-1} - p_k^{\beta-1}) \left( p_k - 2 +  \prod_{j=1, j \neq k}^{r} p_j^{\alpha_j} \right) > 0.
	\end{align*}
\end{lemma}

\begin{proposition}
For any $x \in C_n$ of order $p_k^{\beta}$, $1 \leq k \leq r$, $N([x])$ is a minimum cutset of $\mathcal{P}(C_n)$ if and only if $\beta=\alpha_k$, and $\alpha_k = 2$ or $n$ is a product of two distinct primes.
\end{proposition}

\begin{proof}
Suppose that $N([x])$ is a minimum cutset of $\mathcal{P}(C_n)$. From Lemma \ref{lem_deg_prime_power}, 
\begin{align*}
 |N([x])| & = p_k^{\beta} - 2 \, \phi \left(p_k^{\beta} \right) + (p_k^{\alpha_k} - p_k^{\beta-1}) \prod_{i=1, i \neq k}^{r} p_i^{\alpha_i}\\
 & = (p_k^{\alpha_k} - p_k^{\beta-1}) \left[\prod_{i=1, i \neq k}^{r} p_i^{\alpha_i} \right]  -  \left(p_k^{\beta} - 2 p_k^{\beta-1} \right).
\end{align*}

From Lemma \ref{lem_minimal_cutset}, $Y$ is a cutset $\mathcal{P}(C_n)$ and that
\begin{align*}
	|Y| & = \phi(n) + \frac{n}{p_k} - p_k^{\alpha_k-1} \phi \left(\dfrac{n}{p_k^{\alpha_k}} \right)\\
	& = \phi \left( \prod_{i=1}^{r} p_i^{\alpha_i} \right) + p_k^{\alpha_k-1} \left( \prod_{i=1, i \neq k}^{r} p_i^{\alpha_i} \right) - p_k^{\alpha_k-1} \phi\left( \prod_{i=1, i \neq k}^{r} p_i^{\alpha_i} \right)\\
	& = \left( \prod_{i=1}^{r} p_i^{\alpha_i-1} \right) \left[ \phi\left( \prod_{i=1}^{r} p_i \right) +  \prod_{i=1, i \neq k}^{r} p_i  -  \phi\left( \prod_{i=1, i \neq k}^{r} p_i \right) \right]\\
		& = \left( \prod_{i=1}^{r} p_i^{\alpha_i-1} \right) \left[   \prod_{i=1, i \neq k}^{r} p_i  +  (p_k-2) \phi\left( \prod_{i=1, i \neq k}^{r} p_i \right) \right].
\end{align*}
Hence
\begin{align}\label{eq1.1}
	& \displaystyle |N([x])| - |Y| \nonumber\\ 
	& = \left( \prod_{i=1, i \neq k}^{r} p_i^{\alpha_i-1} \right) \left[ (p_k^{\alpha_k} - p_k^{\beta-1}) \prod_{i=1, i \neq k}^{r} p_i  - p_k^{\alpha_k-1} \prod_{i=1, i \neq k}^{r} p_i  - (p_k^{\alpha_k} - 2p_k^{\alpha_k-1}) \phi\left( \prod_{i=1, i \neq k}^{r} p_i \right) \right] \nonumber \\
	& \qquad \qquad \qquad \qquad \qquad \qquad \qquad \qquad \qquad \qquad \qquad \qquad \qquad \qquad \qquad -  \left(p_k^{\beta} - 2 p_k^{\beta-1} \right) \nonumber\\
  	& = \left( \prod_{i=1, i \neq k}^{r} p_i^{\alpha_i-1} \right) \left[ (p_k^{\alpha_k} - p_k^{\alpha_k-1} - p_k^{\beta-1}) \prod_{i=1, i \neq k}^{r} p_i - (p_k^{\alpha_k} - 2p_k^{\alpha_k-1}) \phi\left( \prod_{i=1, i \neq k}^{r} p_i \right) \right] \nonumber\\
  & \qquad \qquad \qquad \qquad \qquad \qquad \qquad \qquad \qquad \qquad \qquad \qquad \qquad \quad \quad -  \left(p_k^{\beta} - 2 p_k^{\beta-1} \right).
\end{align}

From (\ref{eq1.1}), we get
\begin{align}\label{eq1.2}
	& \displaystyle |N([x])| - |Y| \nonumber\\ 
	& \geq \left( \prod_{i=1, i \neq k}^{r} p_i^{\alpha_i-1} \right) \left[ (p_k^{\alpha_k} - 2p_k^{\alpha_k-1}) \prod_{i=1, i \neq k}^{r} p_i - (p_k^{\alpha_k} - 2p_k^{\alpha_k-1}) \phi\left( \prod_{i=1, i \neq k}^{r} p_i \right) \right] \nonumber\\
	& \qquad \qquad \qquad \qquad \qquad \qquad \qquad \qquad \qquad \qquad \qquad \qquad \qquad \qquad -  \left(p_k^{\beta} - 2 p_k^{\beta-1} \right)\\
  	& = \left( \prod_{i=1, i \neq k}^{r} p_i^{\alpha_i-1} \right) p_k^{\alpha_k-1} (p_k - 2) \left[  \prod_{i=1, i \neq k}^{r} p_i -  \phi\left( \prod_{i=1, i \neq k}^{r} p_i \right) \right] -  p_k^{\beta-1} (p_k - 2)\\
    & = \left( \prod_{i=1, i \neq k}^{r} p_i^{\alpha_i-1} \right) p_k^{\beta-1} (p_k - 2) \left[ p_k^{\alpha_k - \beta} \left\{  \prod_{i=1, i \neq k}^{r} p_i -  \phi\left( \prod_{i=1, i \neq k}^{r} p_i \right) \right\} -  1 \right].
\end{align}

We observe that $\prod_{i=1, i \neq k}^{r} p_i -  \phi\left( \prod_{i=1, i \neq k}^{r} p_i \right) \geq 1$ and equality holds if and only if $r=2$. From above, we conclude that $|N([x])| \geq |Y|$.

From (2), if $\beta < \alpha_k$, then $|N([x])| > |Y|$. Since this is a contradiction, $\beta = \alpha_k$. As a result, we get 
\begin{align*}
	 \displaystyle |N([x])| - |Y| \nonumber  & = \left( \prod_{i=1, i \neq k}^{r} p_i^{\alpha_i-1} \right) p_k^{\alpha_k-1} (p_k - 2) \left[ \prod_{i=1, i \neq k}^{r} p_i -  \phi\left( \prod_{i=1, i \neq k}^{r} p_i \right)  -  1 \right].
\end{align*}
 
 If $p_k = p_1 = 2$, then $|N([x])| = |Y|$. Now, let $p_k > 2$. Then $|N([x])| = |Y|$ if and only if $r=2$.
\end{proof}

\begin{lemma}\label{lem3}
	For any positive integer $n$, the cyclic vertex connectivity and vertex connectivity of $\mathcal{P}(C_n)$ are not equal if one the following holds:
	\begin{enumerate}[\rm(i)]
		\item $n$ is a prime factor,
		\item $n = p_1p_2$ for some primes $p_1 < p_2$ such that $p_1 \leq 3$,
		\item $n = 4p$ for some odd prime $p$.
	\end{enumerate}
\end{lemma}

\begin{proof}
By Theorem \ref{thm1}, if (i) or (ii) hold, or if $n = 12$, then $\mathcal{P}(C_n)$ is not cyclically separable. Then $c\kappa(\mathcal{P}(C_n)) = \infty$. As a result, $c\kappa(\mathcal{P}(C_n)) \neq \kappa(\mathcal{P}(C_n))$. Now let $n = 4p$ for some prime $p \geq 5$. Then by Lemma \ref{lem2}, $\kappa(\mathcal{P}(C_n)) = \phi(n) + 2$. Next we compute $c\kappa(\mathcal{P}(C_n))$. Note that  $C'_n$ is the union of $O_{4}$, $O_{2} $, $ O_{2p} $, and  $ O_{p}$, which are cliques of size $2$, $1$, $p-1$, and $p-1$, respectively, in $\mathcal{P}'(C_n)$. The adjacency relation between the elements of these four sets can be visualized in the following figure. We observe that $O_{2} $ and $ O_{2p} $ are the only minimal disconnecting sets of $\mathcal{P}'(C_n)$.
\begin{figure}[h]
	
	\begin{center}
		\begin{tikzpicture}[
			roundnode/.style={circle, draw=black},
			squarednode/.style={rectangle, draw=red!60, fill=red!5, very thick, minimum size=5mm},
			]
			\node[roundnode]      (maintopic)                              {$O_{2p}$};
			\node[roundnode]        (uppercircle)       [left=of maintopic] {$O_{2}$};
			\node[roundnode]      (rightsquare)       [right=of maintopic] {$O_{p}$};
			\node[roundnode]        (lowercircle)       [left=of uppercircle] {$O_{2^{2}}$};
			
			\draw (uppercircle.east) -- (maintopic.west);
			\draw (maintopic.east) -- (rightsquare.west);
			\draw (lowercircle.east) -- (uppercircle.west);
		\end{tikzpicture}
		\caption{$\mathcal{P}'(C_n)$} \label{fig:M2}
	\end{center}	
\end{figure}
We deduce from the figure that $O_1 \cup O_{n} \cup O_{2p}$ is the only cyclic cutset of $\mathcal{P}(C_n)$. Hence $c\kappa(\mathcal{P}(C_n)) = \phi(n) + 1 + \phi(2p)  = \phi(n) + 1 + p-1 = \phi(n) + p$. Hence, again $\kappa(\mathcal{P}(C_n)) \neq c\kappa(\mathcal{P}(C_n))$.
\end{proof}

The finally prove the following theorem.

\begin{theorem}
	For any positive integer $n$, the cyclic connectivity and vertex connectivity of $\mathcal{P}(C_n)$ are equal if and only if the following holds:
	\begin{enumerate}[\rm(i)]
		\item $n$ has at least two prime factors,
		\item $n \neq p_1p_2$ for any primes $p_1 < p_2$ such that $p_1 \leq 3$,
		\item $n \neq 4p$ for any odd prime $p$.
	\end{enumerate}
\end{theorem}

\begin{proof}
	We prove this theorem by taking its contraposition. That is, we prove that the cyclic vertex connectivity and vertex connectivity of $\mathcal{P}(C_n)$ are not equal if one the following holds:
	\begin{enumerate}[\rm(a)]
		\item $n$ is a prime factor,
		\item $n = p_1p_2$ for some primes $p_1 < p_2$ such that $p_1 \leq 3$,
		\item $n = 4p$ for some odd prime $p$.
	\end{enumerate}
	
By Theorem \ref{thm1}, if (a) or (b) hold, then $\mathcal{P}(C_n)$ is not cyclically separable. That is, $c\kappa(\mathcal{P}(C_n)) = \infty$. As a result, $c\kappa(\mathcal{P}(C_n)) \neq \kappa(\mathcal{P}(C_n))$. Now let (c) holds, i.e., $n = 4p$ for some odd prime $p$. Then by Lemma \ref{lem2}, $\kappa(\mathcal{P}(C_n)) = \phi(n) + 2$. Next we compute $c\kappa(\mathcal{P}(C_n))$. Note that  $C'_n$ is the union of $O_{4}$, $O_{2} $, $ O_{2p} $, and  $ O_{p}$, which are cliques of size $2$, $1$, $p-1$, and $p-1$, respectively, in $\mathcal{P}'(C_n)$. The adjacency relation between the elements of these four sets can be visualized in the following figure.
\begin{figure}[h]
	
	\begin{center}
		\begin{tikzpicture}[
			roundnode/.style={circle, draw=black},
			squarednode/.style={rectangle, draw=red!60, fill=red!5, very thick, minimum size=5mm},
			]
			\node[roundnode]      (maintopic)                              {$O_{2p}$};
			\node[roundnode]        (uppercircle)       [left=of maintopic] {$O_{2}$};
			\node[roundnode]      (rightsquare)       [right=of maintopic] {$O_{p}$};
			\node[roundnode]        (lowercircle)       [left=of uppercircle] {$O_{2^{2}}$};
			
			\draw (uppercircle.east) -- (maintopic.west);
			\draw (maintopic.east) -- (rightsquare.west);
			\draw (lowercircle.east) -- (uppercircle.west);
		\end{tikzpicture}
		\caption{$\mathcal{P}'(C_n)$} \label{fig:M1}
	\end{center}	
\end{figure}
We deduce from the figure that $O_1 \cup O_{n} \cup O_{2p}$ is the only cyclic cutset of $\mathcal{P}(C_n)$. Hence $c\kappa(\mathcal{P}(C_n)) = \phi(n) + 1 + \phi(2p)  = \phi(n) + 1 + p-1 = \phi(n) + p$. Hence, again $\kappa(\mathcal{P}(C_n)) \neq c\kappa(\mathcal{P}(C_n))$.

 Conversely, let $c\kappa(\mathcal{P}(C_n)) \neq \kappa(\mathcal{P}(C_n))$. If $\mathcal{P}(C_n)$ is a complete graph, then $n$ is a prime power. Now, suppose that $\mathcal{P}(C_n)$ is not a complete graph. Then $\mathcal{P}(C_n)$ has a minimum cutset $X$ such that $\kappa(\mathcal{P}(C_n)) = |{X}|$, and that $\mathcal{P}(C_n \setminus X)$ has at most one component containing a cycle.

 Let $\Gamma_0$ be a component of $\mathcal{P}(C_n \setminus X)$ with no cycle. We know that for any positive divisor $d$ of $n$, $C_n$ has exactly $\phi(d)$ elements of order $d$ and that those $\phi(d)$ elements form a clique in $ \mathcal{P}(C_n) $. 
 
 Let $x \in V(\Gamma_0)$. Note that $x \neq e$. Then $\phi(\text{o}(x)) \leq 2$, because, otherwise $[x]$ will induce a clique of size at least three.
  If $\text{o}(x)$ has a prime divisor $p \geq 5$, then $\text{o}(x) \geq 5$. So, $\phi(\text{o}(x)) \geq 4$, which is not possible. Then we can write $\text{o}(x)  = 2^k 3^l$, where $k\geq 0$, $l\geq 0$, and $k+l \neq 0$. If $k=0$, then $\phi(\text{o}(x))  = \phi(3^l) = 3^{l-1}2$. As $\phi(\text{o}(x)) \leq 2$, this implies $l-1 = 0$, so that $\text{o}(x)  = 3$. Whereas, if $l=0$, then $\phi(\text{o}(x))  = \phi(2^k) = 2^{k-1}$. As $\phi(\text{o}(x)) \leq 2$, this implies $k-1 = 0$ or $k-1 = 1$. So that $\text{o}(x)  = 2$ or $\text{o}(x)  = 4$. Finally, let $k \neq 0$ and $l \neq 0$. Then $\phi(\text{o}(x))  = \phi(2^k 3^l) = 2^{k-1}3^{l-1}2 = 2^{k}3^{l-1}$. Again as $\phi(\text{o}(x)) \leq 2$, we have $k = 1$ and $l-1 = 0$. Hence, $\text{o}(x)  = 6$. Thus, the possible orders of $x$ are $2,3,4$, and $6$.  
  
  Note that $\Gamma_0$ is a connected subgraph of $\mathcal{P}(C_n)$. If possible, let $x,y \in V(\Gamma_0)$, $x \sim y$, and that $\langle x \rangle \neq \langle y \rangle$. Then $[x] \cup [y]$ induces a clique of size at least three in $V(\Gamma_0)$. This contradicts the fact that $\Gamma_0$ has no cycles. As a result, we have the following cases.

 \medskip
 \noindent
 \textbf{Case 1.} $V(\Gamma_0)$ has $\phi(2)=1$ element of order $2$. Then $G$ is of even order. Let $V(\Gamma_0) = \{x\}$. Then $x$ is the unique element of order two in $G$. Then $N(x) \subseteq X$ as $\Gamma_0$ is a component of  $\mathcal{P}(C_n \setminus X)$ with only vertex $x$. However, as $N(x)$ is itself a cutset of $\mathcal{P}(C_n)$ and $X$ is a minimum cutset of $\mathcal{P}(C_n)$, we have $X = N(x)$. 

 Thus, by Lemma \ref{lem_deg_prime_power},
 \begin{align*}
 	|X| = |N(x)| = (2^{\alpha_1}-1)\, p_2^{\alpha_2} \cdots p_r^{\alpha_r}.
\end{align*}

Next let $y$ be an element of order $2^{\alpha_1}$. Then 
\begin{align*}
	|N([y])| & = 2^{\alpha_1} - 2\, \phi(2^{\alpha_1}) +  (2^{\alpha_1}-2^{\alpha_1-1}) p_2^{\alpha_2} \cdots p_r^{\alpha_r} = (2^{\alpha_1}-2^{\alpha_1-1}) p_2^{\alpha_2} \cdots p_r^{\alpha_r}.
\end{align*}
\begin{align*}
	|X| -	|N([y])|  & =  (2^{\alpha_1}-1)\, p_2^{\alpha_2} \cdots p_r^{\alpha_r} - (2^{\alpha_1}-2^{\alpha_1-1}) p_2^{\alpha_2} \cdots p_r^{\alpha_r}\\
	& =  (2^{\alpha_1-1}-1)\, p_2^{\alpha_2} \cdots p_r^{\alpha_r}.
\end{align*}
If $\alpha_1 > 1$, then $2^{\alpha_1-1} > 1$, So, $|X| > |N([y])|$, contradicting the fact that $X$ is a miniumum cutset of $\mathcal{P}(C_n)$. So $\alpha_1 = 1$.

From Lemma \ref{lem_alpha}, $\alpha(n)$ is an upper bound of $\mathcal{P}(C_n)$.
\begin{align*}
	& |X| - \alpha(n)\\  & =   \prod_{j=2}^{r} p_j^{\alpha_j} - \left[ \phi(n) + \dfrac{n}{p_r} - p_r^{\alpha_{r}-1}\phi\left( \dfrac{n}{p_r^{\alpha_{r}}} \right) \right]\\
	& = p_2^{\alpha_2} \cdots p_r^{\alpha_r} - \phi(p_2^{\alpha_2} \cdots p_r^{\alpha_r}) - 2\,p_2^{\alpha_2} \cdots p_{r-1}^{\alpha_{r-1}}p_r^{\alpha_r-1} + \phi(p_2^{\alpha_2} \cdots p_{r-1}^{\alpha_{r-1}})p_r^{\alpha_r-1}\\
	& = p_2^{\alpha_2-1} \cdots p_r^{\alpha_r-1} \left[ p_2 \cdots p_r - \phi(p_2 \cdots p_r) - 2\,p_2 \cdots p_{r-1} + \phi(p_2 \cdots p_{r-1}) \right]\\
	& = p_2^{\alpha_2-1} \cdots p_r^{\alpha_r-1} \left[ p_2 \cdots p_{r-1} (p_r - 2) - \phi(p_2 \cdots p_{r-1}) (p_r - 2) \right]\\
	& = p_2^{\alpha_2-1} \cdots p_r^{\alpha_r-1} (p_r - 2) \left[ p_2 \cdots p_{r-1}  - \phi(p_2 \cdots p_{r-1}) \right].
\end{align*}

Hence, if $r \geq 3$, then $ |X| > \alpha(n) $. Since this is a contradiction, we have $r = 2$.

So, we can write $n = 2 p_2^{\alpha_2}$. If possible, suppose that $\alpha_2 \geq 2$. By Lemma \ref{lem_minimal_cutset}, $Y := O_n \cup H_{p_2^{\alpha_2-1}}$ is a minimum cutset of $\mathcal{P}(C_n)$. We have $C_n - Y = O_{p_2^{\alpha_2}} \cup \left( \displaystyle\bigcup_{i=0}^{\alpha_2-1} O_{2 p_2^{i}} \right)$. We observe that no element of $O_{p_2^{\alpha_2}} $ is adjacent to any element of $ \displaystyle\bigcup_{i=0}^{\alpha_2-1} O_{2 p_2^{i}} $ in $\mathcal{P}(C_n) - Y$. As $\alpha_2 \geq 2$, both $O_{p_2^{\alpha_2}}$ and $ \displaystyle\bigcup_{i=0}^{\alpha_2-1} O_{2 p_2^{i}} $ induce cliques of sizes at least three. Hence, $Y$ is a cyclic cutset of $\mathcal{P}(C_n)$. Whereas, because $Y$ is a minimum cutset $|Y| = |X|$. So, $c\kappa(\mathcal{P}(C_n)) = \kappa(\mathcal{P}(C_n))$, a contradiction. Therefore, we conclude that $n = 2 p_2$.

 \medskip
\noindent
\textbf{Case 2.} $\Gamma_0$ has $\phi(4) = 2$ vertices of order $4$. Let $x$ be an element of order $4$ in $C_n$. Then $\Gamma_0$ is induced by $[x]$, and that $N([x]) \subseteq X$. However, as $ N([x]) $ is itself a cutset of $\mathcal{P}(C_n)$ and $X$ is a minimum cutset of $\mathcal{P}(C_n)$, we have $X = N([x])$. 
If possible, let $ \alpha_1 > 2$. Then for any $y \in C_n$ of order $2^{\alpha_1}$,  $|N([x])| > |N([y])|$ by Lemma \ref{compare_nbd}. As this is a contradiction, we have $\alpha_1 = 2$. Thus, by Lemma \ref{lem_deg_prime_power},
\begin{align*}
	|X| = N([x]) = 2 p_2^{\alpha_2} \cdots p_r^{\alpha_r}.
\end{align*}

From Lemma \ref{lem_alpha}, $\alpha(n)$ is an upper bound of $\mathcal{P}(C_n)$. If $r \geq 3$, then
\begin{align*}
	& |X| - \alpha(n)\\  & =   2\prod_{j=2}^{r} p_j^{\alpha_j} - \left[ \phi(n) + \dfrac{n}{p_r} - p_r^{\alpha_{r}-1}\phi\left( \dfrac{n}{p_r^{\alpha_{r}}} \right) \right]\\
	& = 2p_2^{\alpha_2} \cdots p_r^{\alpha_r} - 2\phi(p_2^{\alpha_2} \cdots p_r^{\alpha_r}) - 4\,p_2^{\alpha_2} \cdots p_{r-1}^{\alpha_{r-1}}p_r^{\alpha_r-1} + 2\phi(p_2^{\alpha_2} \cdots p_{r-1}^{\alpha_{r-1}})p_r^{\alpha_r-1}\\
	& = 2p_2^{\alpha_2-1} \cdots p_r^{\alpha_r-1} \left[ p_2 \cdots p_r - \phi(p_2 \cdots p_r) - 2\,p_2 \cdots p_{r-1} + \phi(p_2 \cdots p_{r-1}) \right]\\
	& = 2p_2^{\alpha_2-1} \cdots p_r^{\alpha_r-1} \left[ p_2 \cdots p_{r-1} (p_r - 2) - \phi(p_2 \cdots p_{r-1}) (p_r - 2) \right]\\
	& = 2p_2^{\alpha_2-1} \cdots p_r^{\alpha_r-1} (p_r - 2) \left[ p_2 \cdots p_{r-1}  - \phi(p_2 \cdots p_{r-1}) \right].
\end{align*}

Hence, if $r \geq 3$, then $ |X| > \alpha(n) $. Since this is a contradiction, we have $r = 2$.

So, we can write $n = 4 p_2^{\alpha_2}$. If possible, suppose that $\alpha_2 \geq 2$. By Lemma \ref{lem_minimal_cutset}, $Y := O_n \cup H_{2p_2^{\alpha_2-1}}$ is a minimum cutset of $\mathcal{P}(C_n)$. We have $C_n - Y = O_{p_2^{\alpha_2}} \cup \left( \displaystyle\bigcup_{i=0}^{\alpha_2-1} O_{4 p_2^{i}} \right)$. We observe that no element of $O_{p_2^{\alpha_2}} $ is adjacent to any element of $ \displaystyle\bigcup_{i=0}^{\alpha_2-1} O_{4 p_2^{i}} $ in $\mathcal{P}(C_n) - Y$. As $\alpha_2 \geq 2$, both $O_{p_2^{\alpha_2}}$ and $ \displaystyle\bigcup_{i=0}^{\alpha_2-1} O_{4 p_2^{i}} $ induce cliques of sizes at least three. Hence, $Y$ is a cyclic cutset of $\mathcal{P}(C_n)$. Whereas, because $Y$ is a minimum cutset $|Y| = |X|$. So, $c\kappa(\mathcal{P}(C_n)) = \kappa(\mathcal{P}(C_n))$, a contradiction. Therefore, we conclude that $n = 4 p_2$.

 \medskip
\noindent
\textbf{Case 3.} $\Gamma_0$ has $\phi(3)$ vertices. Let $x$ be an element of order $3$ in $C_n$. Then $\Gamma_0$ is induced by $[x]$, and that $N([x]) \subseteq X$. However, as $ N([x]) $ is itself a cutset of $\mathcal{P}(C_n)$ and $X$ is a minimum cutset of $\mathcal{P}(C_n)$, we have $X = N([x])$. Suppose $\alpha = \alpha_1$ if $p_1=3$, and $\alpha = \alpha_2$ if $p_2=3$.
If possible, let $ \alpha > 1$. Then for any $y \in C_n$ of order $3^{\alpha}$,  $|N([x])| > |N([y])|$ by Lemma \ref{compare_nbd}. As this is a contradiction, we have $\alpha = 1$. 

We first consider the subcase when $p_1=2$. Then $n = 2^{\alpha_1} \cdot 3 \cdot p_3^{\alpha_3} \cdots p_r^{\alpha_r}$. If $z$ is an element of order $2^{\alpha_1}$ in $C_n$, then 
\begin{align*}
	|N([x])| -	|N([z])|  & =  3 - 2 \cdot \phi \left(3 \right) + (3-1) 2^{\alpha_1} p_3^{\alpha_3} \cdots p_r^{\alpha_r} - (2^{\alpha_1} - 2^{\alpha_1-1})\, 3 \, p_3^{\alpha_3} \cdots p_r^{\alpha_r}\\
	& =   2^{\alpha_1+1} p_3^{\alpha_3} \cdots p_r^{\alpha_r} -  3 \cdot 2^{\alpha_1-1} p_3^{\alpha_3} \cdots p_r^{\alpha_r} - 1\\
  & =    2^{\alpha_1-1} p_3^{\alpha_3} \cdots p_r^{\alpha_r} - 1.
\end{align*}
So, if $\alpha_1 >1$ or that $r > 2$, then $|N([x])| >	|N([z])| $, a contradiction. Then $\alpha_1 =1$ and $r = 2$. Hence $n = 2 \cdot 3 = 6$.

Next we consider the subcase when $p_1=3$. 

By Lemma \ref{lem_alpha}, $\alpha(n)$ is an upper bound of $\kappa(\mathcal{P}(C_n))$. Then we have
\begin{align*}
	& |X| - \alpha(n)\\  & =  3 - 2 \cdot \phi \left(3 \right) + (3-1) \prod_{j=2}^{r} p_j^{\alpha_j} - \left[ \phi(n) + \dfrac{n}{p_r} - p_r^{\alpha_{r}-1}\phi\left( \dfrac{n}{p_r^{\alpha_{r}}} \right) \right]\\
	& =2\, p_2^{\alpha_2} \cdots p_r^{\alpha_r} - 2\,\phi(p_2^{\alpha_2} \cdots p_r^{\alpha_r}) - 3\,p_2^{\alpha_2} \cdots p_{r-1}^{\alpha_{r-1}}p_r^{\alpha_r-1} + 2\,\phi(p_2^{\alpha_2} \cdots p_{r-1}^{\alpha_{r-1}})p_r^{\alpha_r-1} - 1\\
	& = p_2^{\alpha_2-1} \cdots p_r^{\alpha_r-1} \left[ 2\, p_2 \cdots p_r - 2\phi(p_2 \cdots p_r) - 3\,p_2 \cdots p_{r-1} + 2\phi(p_2 \cdots p_{r-1}) \right] - 1\\
	& = p_2^{\alpha_2-1} \cdots p_r^{\alpha_r-1} \left[ p_2 \cdots p_{r-1} (2\, p_r - 3) - \phi(p_2 \cdots p_{r-1}) (2p_r - 4) \right] -1\\
	& > p_2^{\alpha_2-1} \cdots p_r^{\alpha_r-1} \left[ \phi(p_2 \cdots p_{r-1}) (2\, p_r - 3) - \phi(p_2 \cdots p_{r-1}) (2p_r - 4) \right] -1\\
	& = \phi(p_2^{\alpha_2} \cdots p_{r-1}^{\alpha_{r-1}}) p_r^{\alpha_r-1} -1.
\end{align*}
Thus, if $r \geq 3$ or if $\alpha_r > 1$, then $|X| > \alpha(n)$, a contradiction. Hence $r =2$ and $\alpha_r = 1$, so that $n = 3 p_2$.

\medskip
\noindent
\textbf{Case 4.} $\Gamma_0$ has $\phi(6) = 2$ vertices. In this case, $n > 6$. Let $x$ be an element of order $6$ in $C_n$. Then $\Gamma_0$ is induced by $[x]$, and that $N([x]) \subseteq X$. However, as $ N([x]) $ is itself a cutset of $\mathcal{P}(C_n)$ and $X$ is a minimum cutset of $\mathcal{P}(C_n)$, we have $X = N([x])$. Then $n = 2^{\alpha_1} 3^{\alpha_2} p_3^{\alpha_3} \dots p_r^{\alpha_r}$. We have
\begin{align*}
|X| & = 6 - 2 \cdot \phi \left(6 \right) + \underset{d\, | \, \frac{n}{6}} \sum \phi \left( \frac{n}{d} \right)\\
 &  = 2 +  \underset{d\, | \, {2^{\alpha_1-1} 3^{\alpha_2-1} \prod_{j=3}^{r} p_j^{\alpha_j} }} \sum \phi \left( \frac{p_1^{\alpha_1}p_2^{\alpha_2} \cdots p_r^{\alpha_r}}{d} \right)\\
	& = 2+ \left[ \sum_{i=0}^{\alpha_1 - 1} \phi \left(2^{\alpha_1-i}\right) \right] \left[ \sum_{i=0}^{\alpha_2 - 1} \phi \left(3^{\alpha_2-i}\right) \right] \prod_{j=3}^{r} \left\{ \sum_{i=0}^{\alpha_j} \phi \left({p_j^{\alpha_j-i}}\right) \right\} \\
	& = 2+ (2^{\alpha_1} - 1) (3^{\alpha_2} -1) \prod_{j=3}^{r} p_j^{\alpha_j}.
\end{align*}

If possible, let $r \geq 3$. Then for any $y \in C_n$ of order $2^{\alpha_1}3^{\alpha_2}$, 
\begin{align*}
	N([y]) & = 2^{\alpha_1}3^{\alpha_2} - 2 \cdot \phi \left( 2^{\alpha_1}3^{\alpha_2} \right) + \underset{d\, | \, \frac{n}{2^{\alpha_1}3^{\alpha_2}}} \sum \phi \left( \frac{n}{d} \right)\\
	&  = 2^{\alpha_1}3^{\alpha_2-1} +  \underset{d\, | \, {\prod_{j=3}^{r} p_j^{\alpha_j} }} \sum \phi \left( \frac{p_1^{\alpha_1}p_2^{\alpha_2} \cdots p_r^{\alpha_r}}{d} \right)\\
	& = 2^{\alpha_1}3^{\alpha_2-1} + \phi\left( 2^{\alpha_1}3^{\alpha_2} \right) \prod_{j=3}^{r} \left\{ \sum_{i=0}^{\alpha_j} \phi \left({p_j^{\alpha_j-i}}\right) \right\} \\
	& = 2^{\alpha_1}3^{\alpha_2-1} + \phi\left( 2^{\alpha_1}3^{\alpha_2} \right) \prod_{j=3}^{r} p_j^{\alpha_j}.
\end{align*}

\begin{align*}
	|X| - N([y]) & = 2+ (2^{\alpha_1} - 1) (3^{\alpha_2} -1) \prod_{j=3}^{r} p_j^{\alpha_j} - \left[ 2^{\alpha_1}3^{\alpha_2-1} +  \phi\left( 2^{\alpha_1}3^{\alpha_2} \right) \prod_{j=3}^{r} p_j^{\alpha_j} \right]\\
  & = 2 - 2^{\alpha_1}3^{\alpha_2-1} + [(2^{\alpha_1} - 1) (3^{\alpha_2} -1) - \phi\left( 2^{\alpha_1}3^{\alpha_2} \right)] \prod_{j=3}^{r} p_j^{\alpha_j}.
\end{align*}

Since $(2^{\alpha_1} - 1) (3^{\alpha_2} -1) \geq \phi\left( 2^{\alpha_1}3^{\alpha_2} \right)$, 
\begin{align*}
	|X| - N([y]) & > 2 - 2^{\alpha_1}3^{\alpha_2-1} + 3[(2^{\alpha_1} - 1) (3^{\alpha_2} -1) - \phi\left( 2^{\alpha_1}3^{\alpha_2} \right)]\\	
	& = 2 - 2^{\alpha_1}3^{\alpha_2-1} + 3(  2^{\alpha_1} 3^{\alpha_2} -  2^{\alpha_1} - 3^{\alpha_2} + 1 - 2^{\alpha_1} 3^{\alpha_2-1} ) \\
  & = 5 + 5 \cdot 2^{\alpha_1} 3^{\alpha_2-1}  -  3 \cdot 2^{\alpha_1} - 3^{\alpha_2+1}   \\
 & = 5 + 3^{\alpha_2-1} (5 \cdot 2^{\alpha_1} - 9)   -  3 \cdot 2^{\alpha_1}    \\
 & \geq 5 + 5 \cdot 2^{\alpha_1} - 9   -  3 \cdot 2^{\alpha_1}    \\
 & \geq  2^{\alpha_1+1} - 4.
\end{align*}

As $ 2^{\alpha_1+1} \geq 4$, we get $|X| > N([y])$. Since this is a contradiction, $r = 2$. Accordingly, we have $n = 2^{\alpha_1} 3^{\alpha_2}$. As $n > 6$, we have $\alpha_1 > 1$ or $ \alpha_2 > 1$.

Then
 $$|X| = 2 + (2^{\alpha_1} - 1) (3^{\alpha_2} - 1). $$
If $y$ is an element of order $2^{\alpha_1}$, then
\begin{align*}
	|N([y])| & = (2^{\alpha_1} - 2^{\alpha_1-1})3^{\alpha_2}.
\end{align*}

Thus
\begin{align*}
|X| -	|N([y])| & = 2 + (2^{\alpha_1} - 1) (3^{\alpha_2} - 1) - (2^{\alpha_1} - 2^{\alpha_1-1})3^{\alpha_2}\\
& = 3 - 2^{\alpha_1} - 3^{\alpha_2} + 2^{\alpha_1-1}3^{\alpha_2}\\
& = 3 - 2^{\alpha_1} + 3^{\alpha_2} (2^{\alpha_1-1} - 1)\\
& \geq 3 - 2^{\alpha_1} + 3 (2^{\alpha_1-1} - 1)\\
& = 3 \cdot 2^{\alpha_1-1} - 2^{\alpha_1}\\
& =  2^{\alpha_1-1}.
\end{align*}
As $2^{\alpha_1-1} > 0$, we have $|X| >	|N([y])|$. As, this is again a contradiction, Case 4 is not possible. 
\end{proof}

\subsection{Noncyclic groups}

For $n \geq 3$, the \emph{dihederal group} $D_{2n}$ of order $2n$ is a non-abelain group of order $2n$ with presentation
\begin{equation}\label{DihedralEq}
D_{2n} = \langle a, b \; : \; a^{n} = b^2 = e, \; ab = ba^{-1} \rangle.
\end{equation}

It is known that every element of $D_{2n} {\setminus} \langle a \rangle$ is of the form $a^ib$ for some $0 \leq i \leq n-1$, and that $ \langle a^ib \rangle =   \{ e, a^ib \}  $. In particular, 
$$D_{2n} = \langle a \rangle \cup \bigcup\limits_{ i=  0}^{n-1} \langle a^ib \rangle.$$
From the structure of $D_{2n}$, 
\begin{equation}\label{dihedral2}
\mathcal{P}(D_{2n}) = \mathcal{P}(\{ e \}) \vee \left[ \mathcal{P}(\langle a \rangle^*) + \mathcal{P}(\{ b \}) + \mathcal{P}(\{ ab \}) + \dots + \mathcal{P}(\{ a^{n-1}b \}) \right].
\end{equation}

This implies
$$\mathcal{P}(D_{2n}^*) =  \mathcal{P}(\langle a \rangle^*) + \mathcal{P}(\{ b \}) + \mathcal{P}(\{ ab \}) + \dots + \mathcal{P}(\{ a^{n-1}b \}),$$ and hence
\begin{align}\label{dihedral1}
\mathcal{P}(D_{2n}^*) \cong  \mathcal{P}(C_n^*) + nK_1.
\end{align}

\begin{theorem}
	\label{thmdihedral}
	For any integer $n \geq 3$, $\mathcal{P}(D_{2n})$ is cyclically separable if and only if the following conditions hold:
	\begin{enumerate}[\rm(i)]
		\item $n$ has at least two prime factors,
		\item $n \neq p_1p_2$ for any primes $p_1 < p_2$ such that $p_1 \leq 3$,
		\item $n \neq 12$.
	\end{enumerate} 
\end{theorem}

\begin{proof}
In view of Theorem \ref{thm1}, it is enough to prove that $\mathcal{P}(D_{2n})$ is cyclically separable if and only if $\mathcal{P}(C_n)$ is cyclically separable. In this proof, we refer to  the presentation (\ref{DihedralEq}) of $D_{2n}$.  From (\ref{dihedral1}), $\mathcal{P}(D_{2n}^*)$ is a disconnected graph and the components are $ \mathcal{P}(\langle a \rangle^*) $ and $n$ isolated vertices. 

We first assume that $\mathcal{P}(C_n)$ is cyclically separable. Then, as $\mathcal{P}(C_n)$ and $ \mathcal{P}(\langle a \rangle) $  are graph isomorphic, $\mathcal{P}(\langle a \rangle)$ is also cyclically separable. Thus $ \mathcal{P}(\langle a \rangle) $ has a cyclic cutset $S$. So, $\mathcal{P}(D_{2n}) - S$ is disconnected and has at least two components containing cycles. Hence, $\mathcal{P}(D_{2n})$ is cyclically separable.

Conversely, let $\mathcal{P}(D_{2n})$ be cyclically separable. Then it has a cyclic cutset, say $T$. Thus $\mathcal{P}(D_{2n}) - T = \mathcal{P}(\langle a \rangle^*\setminus T) + \mathcal{P}(\{ b, ab, \dots, a^{n-1}b \}\setminus T).$ We notice that $\mathcal{P}(\{ b, ab, \dots, a^{n-1}b \}\setminus T)$ is either empty or consists entirely of isolated vertices. Hence, as $\mathcal{P}(D_{2n}) - T$ has at least two components containing cycles, $\mathcal{P}(\langle a \rangle^*\setminus T)$ must contain at least two components containing cycles. Since $\mathcal{P}(\langle a \rangle\setminus T) = \mathcal{P}(\langle a \rangle^*\setminus T)$ and $\mathcal{P}(C_n)$ and $ \mathcal{P}(\langle a \rangle) $  are graph isomorphic, $\mathcal{P}(C_n)$ is cyclically separable.
 \end{proof}

For any positive integer $n \geq 2$, the \emph{dicyclic group} $Q_{4n}$ is a finite group of order $4n$ having presentation
\begin{equation}\label{DicyclicEq}
	Q_{4n}=\left \langle a, b \mid a^{2n}=e, a^n=b^2, ab=ba^{-1} \right \rangle,
\end{equation}
where $e$ is the identity element of $Q_{4n}$.

We first show by induction that $(a^ib)^2=a^n$ for all $0 \leq i \leq 2n-1$. As $b^2=a^n$, it is trivially true for $i=0$. Let it be true for $i=k$, where $0 \leq k \leq 2n-2$. Then for $i=k+1$, $(a^{k+1}b)^2=a^{k+1}ba^{k+1}b=a^{k}ba^{-1}a^{k+1}b=(a^kb)^2=a^n$, by induction hypothesis.

Now for any $0 \leq i \leq n-1$, we have $(a^ib)^3=a^na^ib=a^{n+i}b$ and $(a^{n+i}b)^3=a^na^{n+i}b=a^ib$. Thus
\begin{equation}\label{EqQn}
	\langle a^ib \rangle=\langle a^{n+i}b \rangle=\{ e, a^ib, a^n, a^{n+i}b \} \text{ for all } 0 \leq i \leq n-1.
\end{equation}

Therefore, any element of $Q_{4n} - \langle a \rangle$ can be written as $a^ib$ for some $0 \leq i \leq 2n-1$. Subsequently, we have
\begin{equation}\label{EqQn2}
	Q_{4n}=\langle a \rangle \cup \bigcup \limits_{i=0}^{n-1}\langle a^ib \rangle
\end{equation}

\begin{theorem}
	\label{thmdicyclic}
	For any integer $n \geq 3$, $\mathcal{P}(Q_{4n})$ is cyclically separable if and only if the following conditions hold:
	\begin{enumerate}[\rm(i)]
		\item $n$ is not a power of $2$,
		\item $n$ is not a prime number,
		\item $n \neq 6$.
	\end{enumerate} 
\end{theorem}

\begin{proof}
We first observe that $n$ is not a power of $2$ if and only if $n$ has a prime factor $p > 2$, and this is equivalent to the statement that $2n$ has at least two prime factors. It is trivial to see that $n$ is not a prime number if and only if  $2n \neq 2p$ for any prime $p > 2$, and that $n \neq 6 \Leftrightarrow 2n \neq 12$. Thus, in view of Theorem \ref{thm1}, $\mathcal{P}(C_{2n})$ is cyclically separable if and only if the following conditions hold:
	\begin{enumerate}[\rm(1)]
		\item $n$ is not a power of $2$,
		\item $n$ is not a prime number,
		\item $n \neq 6$.
	\end{enumerate}
	So, to prove the theorem, it is enough to show that $\mathcal{P}(Q_{4n})$ is cyclically separable if and only if $\mathcal{P}(C_{2n})$ is cyclically separable.  In this proof, we refer to  the presentation (\ref{DicyclicEq}) of $Q_{4n}$. 
	 We first assume that $\mathcal{P}(C_{2n})$ is cyclically separable. Since $\langle a \rangle$ is a cyclic subgroup of order $2n$, $\mathcal{P}(\langle a \rangle)$ is also is cyclically separable. Let $S$ be a cyclic cutset of $\mathcal{P}(\langle a \rangle)$. Since $e$ and $a^n$ are adjacent to all vertices of $\mathcal{P}(\langle a \rangle)$, we have  $\{e, a^n\} \subset S$. Hence $\mathcal{P}(Q_{4n}) - S$ is disconnected and 
	 $$\mathcal{P}(Q_{4n}) - S = (\mathcal{P}(\langle a \rangle) - S) + (\mathcal{P}(\{b, ab, \dots, a^{2n-1}b\}\setminus S)).$$ 
	 Since $\mathcal{P}(\langle a \rangle) - S$ has at least two components containing cycles, so does $\mathcal{P}(Q_{4n}) - S$. Hence $\mathcal{P}(Q_{4n})$ is cyclically separable.
	 
	 Next, to prove the converse, let $\mathcal{P}(Q_{4n})$ be cyclically separable. Let $T$ be a cyclic cutset of $\mathcal{P}(Q_{4n})$. Since $e$ and $a^n$ are adjacent to all vertices of $\mathcal{P}(Q_{4n})$, we have  $\{e, a^n\} \subset T$. In fact, we have 
	 $$\mathcal{P}(Q_{4n}) - T = (\mathcal{P}(\langle a \rangle) - T) + (\mathcal{P}(\{b, ab, \dots, a^{2n-1}b\}\setminus T)).$$ 
	 We observe that $\mathcal{P}(\{b, ab, \dots, a^{2n-1}b\}\setminus T)$ contains no cycles. In fact, it is either empty, or disjoint union of $K_1$'s or $K_2$'s or either. Thus, since $\mathcal{P}(Q_{4n}) - T$ has at least two components containing cycles, so does $\mathcal{P}(\langle a \rangle) - T$. Then $T \cap \langle a \rangle$ is cyclic cutset of $\mathcal{P}(\langle a \rangle)$, and so it is cyclically separable. Finally, since $\langle a \rangle$ is a cyclic subgroup of order $2n$, more generally, $\mathcal{P}(C_{2n})$ is cyclically separable.
\end{proof}

\begin{theorem}
	If $G$ is a dihedral or dicyclic group, then $\kappa(\mathcal{P}(G)) \neq c\kappa(\mathcal{P}(G))$.
\end{theorem}

\begin{proof}
We begin with $D_{2n}$. From (\ref{dihedral2}) (also, see \cite{chattopadhyay2014}), $\{e\}$ is the cutset of  $\mathcal{P}(D_{2n})$. In particular, $\kappa(\mathcal{P}(D_{2n})) = 1$. From (\ref{dihedral1}),
\begin{align}
\mathcal{P}(D_{2n})- \{e\} \cong  (\mathcal{P}(C_n)- \{e\}) + nK_1.
\end{align}
Since the generators of $C_n$ are adjacent to all other vertices in $\mathcal{P}(C_n)$, $\mathcal{P}(C_n)- \{e\}$ is connected. Thus $\mathcal{P}(D_{2n})- \{e\}$ have at most one component containing cycles. This along with Theorem \ref{thmdihedral} implies that $c\kappa(\mathcal{P}(D_{2n})) > 1$. Hence $\kappa(\mathcal{P}(D_{2n})) \neq c\kappa(\mathcal{P}(D_{2n}))$.

We next consider $Q_{4n}$. It was shown shown in \cite{chattopadhyay2014} that $\{e, a^n \}$ is the minimum cutset of $\mathcal{P}(Q_{4n})$, and so $\kappa(\mathcal{P}(Q_{4n})) = 2$. We have

$$\mathcal{P}(Q_{4n}) - \{e, a^n \} = (\mathcal{P}(\langle a \rangle) - \{e, a^n \}) + \mathcal{P}(\{b, ab, \dots, a^{2n-1}b\}).$$ 

Since the generators of $\langle a \rangle$ are adjacent to all other vertices in $\mathcal{P}(\langle a \rangle)$, $\mathcal{P}(\langle a \rangle) - \{e, a^n \}$ is connected. Whereas, $\mathcal{P}(\{b, ab, \dots, a^{2n-1}b\}) \cong n K_2$. Hence $\mathcal{P}(Q_{4n}) - \{e, a^n \}$ has at most one component containing cycles. This together with Theorem \ref{thmdicyclic} implies that $c\kappa(\mathcal{P}(Q_{4n})) > 2$. Hence $\kappa(\mathcal{P}(Q_{4n})) \neq c\kappa(\mathcal{P}(Q_{4n}))$.
\end{proof}

\end{document}